\documentclass[12pt]{amsart}
\usepackage[all]{xy}
\usepackage{amssymb}
\usepackage{amsthm}
\usepackage{hyperref}
\hypersetup{colorlinks=true,linkcolor=blue,citecolor=magenta}
\usepackage{amsmath}
\usepackage{amscd,enumitem}
\usepackage{verbatim}
\usepackage{eurosym}
\usepackage{float}
\usepackage{color}
\usepackage{dcolumn}
\usepackage[mathscr]{eucal}
\usepackage[all]{xy}
\usepackage{hyperref}
\usepackage{bbm}
\usepackage[textheight=8.3in, textwidth=6.5in]{geometry}
\newtheorem*{thm*}{Theorem}
\newtheorem*{conj*}{Conjecture}

\newtheorem*{remark}{Remark}
\newtheorem{thm}{Theorem}[section]
\newtheorem{cor}[thm]{Corollary}

\newtheorem{prop}[thm]{Proposition}

\newtheorem*{rmk}{Remark}

\newcommand{\Z}{\mathbb{Z}}

\newcommand{\SL}{\operatorname{SL}}
\newcommand{\C}{\mathbb{C}}
\numberwithin{equation}{section}

\begin{document}
\title{On Witten's extremal partition functions}
\dedicatory{In celebration of George Andrews' 80th birthday}
\author{Ken Ono and Larry Rolen}
\address{Department of Mathematics and Computer Science, Emory University, Atlanta, GA 30322}
\email{ken.ono@emory.edu}
\address{Department of Mathematics, Vanderbilt University, 1326 Stevenson Center, Nashville, TN 37240}
\email{larry.rolen@vanderbilt.edu}

\begin{abstract}
In his famous 2007 paper on three dimensional quantum gravity, Witten
 defined candidates for the partition functions $$Z_k(q)=\sum_{n=-k}^{\infty}w_k(n)q^n$$ of potential extremal conformal field theories (CFTs) with central charges of the form $c=24k$. Although such CFTs remain elusive, he proved that these modular functions are well-defined. In this note, we point out several explicit representations of these functions. These involve the partition function $p(n)$, Faber polynomials, traces of singular moduli, and Rademacher sums. 
Furthermore,
for each prime $p\leq 11$, the $p$ series $Z_k(q)$, where $k\in \{1, \dots, p-1\} \cup \{p+1\},$ 
possess a Ramanujan congruence. More precisely,
for every non-zero integer $n$ we have that
$$
      w_k(pn) \equiv 0\begin{cases} \pmod{2^{11}}\ \ \ \ &{\text {\rm if}}\ p=2,\\
      \pmod{3^5} \ \ \ \ &{\text {\rm if}}\ p=3,\\
      \pmod{5^2}\ \ \ \ &{\text {\rm if}}\ p=5,\\
      \pmod{p} \ \ \ \ &{\text {\rm if}}\ p=7, 11.
      \end{cases}
      $$
 \end{abstract}
\maketitle
\section{Introduction and statement of results}

In \cite{Witten3D}, Witten defined a sequence of functions which he proposed encode quantum states of three-dimensional gravity. Namely, he purported the existence of an extremal conformal field theory (in the language of \cite{Hoehn}) at any central charge $c=24k$ with $k\geq1$.  They should have partition functions equal to the unique weakly holomorphic\footnote{A modular form on $\SL_2(\Z)$ is said to be weakly holomorphic if its poles (if any) are
supported at the cusp $i\infty$.} modular functions on $\operatorname{SL}_2(\Z)$ with principal part (i.e., the negative powers of $q$ together with the constant term at $i\infty$) determined by:
\begin{equation}\label{Zkdefn}
Z_k(q)=\sum_{n=-k}^{\infty}w_k(n)q^n:=q^{-k}\prod_{n\geq2}\frac{1}{1-q^n}+O(q)
.
\end{equation}
For positive integers $k$, these functions are Witten's candidates for the generating functions that count the quantum states of three dimensional gravity in spacetime asymptotic to AdS$_3$ (see Section 3.1 of \cite{Witten3D}).

It is well-known that the Hauptmodul for $\operatorname{SL}_2(\Z)$, given by (here $\sigma_k(n):=\sum_{d|n}d^k$ and $q:=e^{2\pi i\tau}$)
\[
J(\tau)=j(\tau)-744=\frac{\left(1+240\sum_{n\geq1}\sigma_3(n)q^n\right)^3}{q\prod_{n\geq1}\left(1-q^n\right)^{24}}-744=q^{-1}+196884q+\ldots,
\]
generates the vector space of modular functions over $\C$. In particular, since there are no non-constant holomorphic modular functions, any polynomial in $J(\tau)$ which matches the principal part of $Z_k(q)$ is identically equal to $Z_k(q)$. For instance, we have 
\[
Z_1(q)=J(\tau),\quad\quad Z_2(q)=J^2(\tau)-393767, \quad\quad Z_3(q)=J^3(\tau)-590651J(\tau)-64481279.
\]

Witten gave an elementary argument that proves that the $Z_k(q)$ are well-defined. 
We offer several formulas for the modular functions $Z_k(q)$ in different guises. These formulas rely on expressions for the {\it partition function} $p(n)$, which counts the number of integer partitions of $n$, Faber polynomials, and Rademacher expansions, which are all standard in number theory. The hope is that these expressions might
shed light on the search for these extremal CFTs.

Our first interpretation of the functions $Z_k(q)$ uses the following generating function, which encodes the classical Faber polynomials $F_d(X)$, and where each coefficient of $q^d$ is a monic degree $n$ polynomial in $X$ (see \cite{AKN, ZagierTraces}):
\begin{equation}\label{OmegaDefn}
\begin{aligned}
\Omega(X; \tau)
&:=\frac{1-24\sum_{n\geq1}\sigma_{13}(n)q^n}{q\prod_{n\geq1}\left(1-q^n\right)^{24}}\cdot\frac{1}{J(\tau)-X}
=\sum_{d\geq0}F_d(X)q^d
\\
&=
1+(X-744)q+(X^2-1488X+159768)q^2+\ldots
.
\end{aligned}
\end{equation}
The Faber polynomials can be used to build the unique weakly holomorphic modular functions $J_d(\tau)$ satisfying $J_d(\tau)=q^{-d}+O(q)$ (see \cite{AKN,ZagierTraces}). More precisely, they satisfy 
\begin{equation}
J_d(\tau)=F_d(j(\tau)).
\end{equation}
Our first result is then the following, where
for any $q$-series 
\[f(q)=a_{-m}q^{-m}+\ldots+a_{-1}q^{-1}+a_0+a_1q+\ldots,\]
 we define the following ``principal part'' operator by:
\[
\mathrm{PP}(f)(q)=a_{-m}X_m+\ldots+a_{-1}X_1+a_0\cdot X_0.
\]
\begin{remark}
The principal part operator PP can be thought of as the complement of MacMahon's $\Omega_{\geq }$ operator. For instance, the reader is also referred to \cite{APR}.
\end{remark}
\begin{thm}
\label{Theorem1}
If $k$ is a positive integer, then the following are true.
\begin{enumerate}
\item[\emph{(i)}]
In terms of the partition function $p(n)$, we have
\begin{equation*}
Z_k(q)=p(k)+
\left(J_{k}(q)
-J_{k-1}(q)\right)
+
\sum_{n=1}^{k-1}
p(n)\left(J_{k-n}(\tau)-J_{k-n-1}(\tau)\right)
.
\end{equation*}
\item[\emph{(ii)}]
If we define $\Omega_k(X_0, X_1,\dots)$ by 
\[
\Omega_k(X_0, X_1,\dots):=\mathrm{PP}\left (q^{-k}\prod_{n\geq2}\frac{1}{(1-q^n)}\right)
\]
then
\begin{equation*}
Z_k(q)=\Omega_k(J_0(\tau), J_1(\tau), J_2(\tau),\dots).
\end{equation*}
\end{enumerate}
\end{thm}
The first formula relies only on elementary properties of the partition generating function, while, as we shall see, the second connects the function $Z_k(q)$ to the world of the Monster and its moonshine. 

Our second main result writes $Z_k(q)$ as a blend of values of a canonical non-holomorphic modular function $\mathcal P(\tau)$ at CM points corresponding to elements of class groups (see Section~\ref{PofnAlgFormula} for the definition of $\mathcal P(\tau)$ and the sums of values of $\mathcal P(\tau)$ denoted by $\operatorname{Tr}(\mathcal P;n)$), and Rademacher sums $R_k(\tau)$ (see Section~\ref{PoincareSection} for the precise definitions). 
Specifically, we will express the partition numbers as traces of singular moduli, that is, sums of values at CM points, of the special non-holomorphic modular function $\mathcal P(\tau)$. The definition of 
this function, the exact notion of traces of singular moduli we require here, and the Rademacher series in the following theorem will be given in Section~\ref{PrelimSection}.
\begin{cor}
\label{Theorem2}
We have the following identity:
\[
Z_k(q)=\frac{1}{24k-1}\operatorname{Tr}(\mathcal P;k)+\left(R_k(\tau)-R_{k-1}(\tau)\right)+\sum_{n=1}^{k-1}\frac{1}{24n-1}\operatorname{Tr}(\mathcal P;n)\left(R_{k-n}(\tau)-R_{k-n-1}(\tau)\right).
\]
\end{cor}
\begin{remark}
The results in \cite{BruinierOnoSutherland} indicate how to efficiently compute $p(n)$ as traces of singular moduli numerically, and may be useful to those wishing to implement the 
identities presented here.
\end{remark}
\begin{remark}
Connections between class numbers (which count the number of terms in the traces of singular moduli discussed here), their algebraic structures, and black holes were also described recently in
 \cite{BenjaminKachruOnoRolen}. It would be interesting to see if the connection between the results discussed here and in that paper have a deeper connection.
\end{remark}
\begin{remark}
Although the partition numbers $p(n)$ may also be written as Rademacher sums, here we have chosen to highlight their alternative algebraic representations. 
Specifically, a version of Rademacher's famous exact formula for $p(n)$ is:
\[
p(n)= \frac{2 \pi}{(24n-1)^{\frac{3}{4}}} \sum_{k =1}^{\infty}
\frac{A_k(n)}{k}  I_{\frac{3}{2}}\left( \frac{\pi
\sqrt{24n-1}}{6k}\right),
\]
where 
\[
A_k(n):=\frac{1}{2}\sqrt{\frac{k}{12}}
 \sum_{\substack{d \pmod{24k}\\
d^2\equiv -24n+1\pmod{24k}}} 
\left(\frac{12}{d}\right) e^{\frac{\pi d i}{6k}}
\]
is a Kloosterman sum ($\left(\frac{12}{d}\right)$ denotes a Kronecker symbol)
and $I_{\frac{3}{2}}$ is a modified $I$-Bessel function.
\end{remark}

It turns out that the coefficients of some of the $Z_k(q)=\sum_{n=-k}^{\infty}w_k(n)q^n$ possess striking systematic congruences which are analogous
to the celebrated partition congruences of Ramanujan
\begin{displaymath}
\begin{split}
p(5n+4)&\equiv 0\pmod 5,\\
p(7n+5)&\equiv 0\pmod{7},\\
p(11n+6)&\equiv 0\pmod{11}.
\end{split}
\end{displaymath}
 If $p\leq 11$ is prime, then the $p$ series
$\{Z_1(q),\dots, Z_{p-1}(q)\}\cup \{Z_{p+1}(q)\}$ all simultaneously satisfy Ramanujan congruences modulo
fixed small powers of $p$. Namely, we prove the following theorem.

\begin{thm}\label{congruences}
If $p\leq 11$ is prime and $k\in \{1,\dots, p-1\} \cup \{p+1\}$, then for every non-zero integer $n$ we have that
$$
      w_k(pn) \equiv 0\begin{cases} \pmod{2^{11}}\ \ \ \ &{\text {\rm if}}\ p=2,\\
      \pmod{3^5} \ \ \ \ &{\text {\rm if}}\ p=3,\\
      \pmod{5^2}\ \ \ \ &{\text {\rm if}}\ p=5,\\
      \pmod{p} \ \ \ \ &{\text {\rm if}}\ p=7, 11.
      \end{cases}
      $$
\end{thm}

\begin{rmk}
We have made no attempt to completely classify all of the Ramanujan type congruences satisfied by the $Z_k(q)$.
For each positive integer $m$ and each $k\geq 1$, it is well known that there are arithmetic progressions $an+b$ for which
$$
w_k(an+b)\equiv 0\pmod m.
$$
This follows easily from the theory of $p$-adic modular forms (for example, see Chapter 2 of \cite{CBMS}). 
The unexpected phenomenon here is the uniformity of these congruences among the low index $q$-series for the primes $p\leq 11$. 
\end{rmk}


\section{Nuts and bolts and the proofs}
\label{PrelimSection}
In this section, we review the basic definitions and results needed for the statements and proofs of the main theorems.

\subsection{Faber polynomials and the connection to Monstrous Moonshine}
Recall from above that 
for each $d\geq0$, we have a function $J_d$, which is the unique weakly holomorphic modular function on $\operatorname{SL}_2(\Z)$ with principal part
\begin{equation}
\label{JdPP}
J_d(\tau):=q^{-d}+O(q).
\end{equation}
In particular, $J_0=1$ and for $d\geq1$, 
in terms of the normalized Hecke operators $T_d$ (see \cite{AKN, ZagierTraces}), we have 
\[J_d(\tau)=d\left(J_1(\tau)|T_d\right).
\]
These functions are, of course, monic degree $n$ polynomials in $J(\tau)$. These polynomials are known as {\it Faber polynomials}, and are closely related to the denominator formula for the Monster Lie algebra (for details on moonshine and related subjects, the reader is referred to the excellent exposition in \cite{Gannon}).
Specifically, the denominator formula states that
$$
J(z)-J(\tau)
=e^{-2\pi iz} \prod_{\buildrel{m>0}\over{n\in \Z}}
\left(1-e^{2\pi i m z}e^{2\pi i n \tau}\right)^{c_{mn}},
$$
where $J(\tau)=:\sum_{n\geq-1}c_nq^n$.
This formula plays a key role in the overall proof of moonshine, and in particular in connecting $J$ with the natural infinite dimensional graded module of the monster which moonshine studies.
Equivalently, Asai, Kaneko, and Ninomiya (cf. Theorem 3 of \cite{AKN}) proved the following result for the logarithmic derivative (with respect to $\tau$) of the preceding generating function (see \eqref{OmegaDefn}):
\begin{equation}
\label{JnOmega}
 \sum_{n=0}^{\infty} J_n(\tau)e^{2\pi i nz}=\Omega(j(z); \tau)
 .
\end{equation}

\subsection{An algebraic formula for $p(n)$}\label{PofnAlgFormula}
Here, we recall a finite, algebraic formula for the partition numbers obtained in \cite{BruinierOno2}.
To state this, we first require the quasimodular Eisenstein series 
\[E_2(\tau):=1-24\sum_{n\geq1}\sigma_1(n)q^n\]
and the Dedekind eta function 
\[
\eta(\tau):=q^{\frac{1}{24}}\prod_{n\geq1}\left(1-q^n\right).
,
\]
Then we consider the weight $-2$,
level $6$ modular function
\[
G(\tau):=\frac{1}{2}\frac{E_2(\tau)-2E_2(2\tau)-3E_2(3\tau)+6E_2(6\tau)}{\eta(\tau)^2\eta(2\tau)^2\eta(3\tau)^2\eta(6\tau)^2}
.
\]
Our distinguished non-holomorphic modular function $\mathcal P$ is then obtained by applying a Maass raising operator to $G$:
\[
\mathcal{P}(\tau):=\frac{i}{2\pi}\frac{\partial G}{\partial\tau}-\frac{G(\tau)}{2\pi\operatorname{Im}(\tau)}.
\]
We then require the distinguished collection of  binary quadratic forms given by
\[
\mathcal{Q}_{D,6,1}:=
 \left\{Q=[a,b,c]: a,b,c\in\Z, b^2-4ac= -24D+1, 6|a, a>0,\ b\equiv 1 \pmod{12} \right\}.
\]
For each quadratic form $Q$ in this set, we define the corresponding 
CM point $\tau_Q$ to be the point in the upper half plane satisfying $a\tau_Q^2+b\tau_Q+c=0$. 
Finally, the trace of $\mathcal P$ at the relevant CM points is given by
\[
\operatorname{Tr}(\mathcal P;n):=\sum_{Q\in\mathcal{Q}_{n,6,1}}\mathcal{P}\left(\tau_Q\right)
.
\]
In terms of these notations, the main result of \cite{BruinierOno2} is the following representation for $p(n)$ in terms of these traces.
\begin{thm}\label{PartitionAlgFormula}
For any $n\geq1$, we have
\[
p(n)=\frac{1}{24n-1}\operatorname{Tr}(\mathcal P;n).
\]
Moreover, $(24n-1)\mathcal{P}\left(\tau_Q\right)$ is always an algebraic integer.
\end{thm}

\subsection{Rademacher series}
\label{PoincareSection}

In this section, we recall the required expressions for our Rademacher series. These can be built out of well-known expressions for Poincar\'e series, for example, the reader is referred to Section~6.3 of \cite{BOOK}. However, these formulas here are very classical, and date back to seminal work of Rademacher, Zuckerman, and others. From these classical results, we can write the following Rademacher series representation for 
$J_d(\tau).
$
\begin{prop}
\label{RadSeries}
If $d$ is a positive integer, then
\begin{equation*}
J_d(\tau)=R_d(\tau) = q^{-d} + \sum_{n\geq1} r_{d,n}q^n,
\end{equation*}
where
\[
r_{d,n}
= 2 \pi \sqrt{\frac{d}{n}}
 \times \sum_{c > 0} \frac{K(n;c)}{c}
I_{1}\left(\frac{4\pi \sqrt{dn}}{c}\right),
\]
and where
\[
K(n;c):=
\sum_{ r \pmod c^\times}  \exp\left(2\pi i\left(\frac{-d\overline
r+nr}{c}\right)\right) 
\]
is a Kloosterman sum ($\bar r$ denotes the multiplicative inverse of $r$ modulo $c$) and $I_1$ is a modified $I$-Bessel function.
\end{prop}

\subsection{Proofs of Theorem \ref{Theorem1}, Corollary \ref{Theorem2} and Theorem~\ref{congruences}}
Here we prove the main results of this paper. 

\begin{proof}[Proof of Theorem \ref{Theorem1}]

We begin with part (i). By \eqref{Zkdefn} and \eqref{JdPP}, together with the fact that weakly holomorphic modular functions are determined by their principal parts, we have that 
\begin{equation*}
\begin{aligned}
Z_k(q)&=(1-q)\cdot\sum_{n\geq0}p(n)q^{n-k}+O\left(q\right)
\\
&=
\sum_{n=0}^k\left(p(n)-p(n-1)\right)q^{n-k}+O(q)
\\
&=
p(0)\left(q^{-k}-q^{1-k}\right)+p(1)\left(q^{1-k}-q^{2-k}\right)+\ldots+p(k-1)\left(q^{-1}-q^0\right)+p(k)+O(q)
\\
&
=
p(k)+
\left(J_{k}(\tau)
-J_{k-1}(\tau)\right)
+
\sum_{n=1}^{k-1}
p(n)\left(J_{k-n}(\tau)-J_{k-n-1}(\tau)\right)
.
\end{aligned}
\end{equation*}
The claim in part (ii) follows from part (i) and \eqref{JnOmega}.

\end{proof}

\begin{proof}[Proof of Corollary~\ref{Theorem2}]
Corollary~\ref{Theorem2} follows from Theorem~\ref{PartitionAlgFormula} and Proposition~\ref{RadSeries}.
\end{proof}

\begin{proof}[Proof of Theorem~\ref{congruences}]
Recall that the Atkin $U(p)$-operator is defined by
$$
\left(\sum_{n\gg -\infty} a(n)q^n\right) \ | \ U(p):=
\sum_{n\gg -\infty} a(pn)q^n.
$$
Suppose that $p\leq 11$ is prime.
If $F(X)$ is a monic polynomial with integer coefficients 
and we let
$$
F(j(\tau))=\sum a(n)q^n.
$$
If $p\leq 11$ is prime and $\deg(F(X))<p$, then Theorem 2.3 (2) of \cite{ElkiesOnoYang}   implies that
$$
F(j(\tau)) \ | \ U(p)\equiv a(0)\pmod p.
$$
Theorem~\ref{Theorem1} then implies the result for $p=7$ and $11$ and $1\leq k\leq p-1$.
For the cases where $(p,k)\in \{(7,8), (11,12)\}$, one applies Theorem 2.3 (1) of \cite{ElkiesOnoYang}. 

For the remaining cases where $p\leq 5$ and the modulus of the congruence is power of $p$, one may consider the weight $12kp$ holomorphic modular forms
$
Z_k(q)\cdot \Delta(p \tau)^{kp} 
$
on $\Gamma_0(p)$, where $\Delta(\tau)$ is the usual weight 12 normalized cusp form on $\SL_2(\Z)$,
It suffices to show that
$$
(Z_k(q) \cdot \Delta(p \tau)^{kp}) \ | \ U(p)\equiv 0\begin{cases}
\pmod{2^{11}} \ \ \ \ &{\text {\rm if}}\ p=2,\\
\pmod{3^5} \ \ \ \ &{\text {\rm if}}\ p=3,\\
\pmod{5^2}\ \ \ \ &{\text {\rm if}} \ p=5.\\
\end{cases}
$$
 These congruences are easily confirmed using the well-known theorem
of Sturm (for example, see p. 40 of \cite{CBMS}) which reduces each claim to a finite computation.
In particular, one only needs to check the claimed congruences for the first $kp(p+1)$ terms.
\end{proof}

\end{document}